\documentclass[]{interact}

\usepackage{epstopdf}
\usepackage[caption=false]{subfig}

\usepackage[doublespacing]{setspace}
\usepackage[utf8]{inputenc}
\usepackage[ english]{babel}
\usepackage{textcomp} 
\usepackage{amsmath, amsthm, amsfonts, amssymb}
\usepackage{marvosym}
\usepackage{hyphenat}
\usepackage{hyperref}
\usepackage[numbers,sort&compress]{natbib}
\bibpunct[, ]{[}{]}{,}{n}{,}{,}
\renewcommand\bibfont{\fontsize{10}{12}\selectfont}

\newcommand{\E}{{\mathbb{E}}}
\newcommand{\R}{{\mathbb{R}}}
\newcommand{\D}{{\mathbb{D}}}

\newcommand{\N}{{\mathbb N}}

\newcommand{\ve}{\varepsilon}

\newcommand{\PP}{{\mathbb{P}}}
\theoremstyle{plain}
\newtheorem{thm}{Theorem}[section]
\newtheorem{lem}[thm]{Lemma}
\newtheorem{corollary}[thm]{Corollary}
\newtheorem{prop}[thm]{Proposition}

\theoremstyle{definition}
\newtheorem{defn}[thm]{Definition}
\newtheorem{example}[thm]{Example}

\theoremstyle{remark}
\newtheorem{remark}{Remark}

\begin{document}

\title{Large deviation principle for generalised multiple intersection local times of multidimensional Brownian motion}

\author{
\name{Andrey A. Dorogovtsev \textsuperscript{a}\thanks{CONTACT Andrey A. Dorogovtsev. Email: andrey.dorogovtsev@gmail.com} and Naoufel Salhi\textsuperscript{b}\thanks{CONTACT Naoufel Salhi. Email: salhi.naoufel@gmail.com} }
\affil{\textsuperscript{a}Institute of Mathematics, National Academy of Sciences of Ukraine, Ukraine; \textsuperscript{b}Laboratory of Stochastic Analysis and Applications, Department of Mathematics, Faculty of Sciences of Tunis, University of Tunis El Manar, Tunisia}
}

\maketitle
\begin{center}
\textit{To the memory of professor Habib Ouerdiane}
\end{center}
\begin{abstract}
In this paper we consider examples of positive generalised Wiener functions and we establish a large deviation principle for the generalised multiple intersection local time of the multidimensional Brownian motion.
\end{abstract}


\begin{amscode}
60F10, 60H40, 60J55
\end{amscode}

\section{Introduction}

In the recent work \cite{DS23} the authors considered the double self-intersection local time of a $d$-dimensional Brownian motion $w(t)=\big( w_1(t),\cdots,w_d(t)\big),\; 0\leqslant t\leqslant 1,$ in the case $d\geqslant 4$. Formally written as
\begin{equation}
\label{example1}
\rho_2(u)=\int _{\Delta_2} \delta_0 \big(w(t)-w(s)-u \big)  dsdt 
\end{equation} 
where $\Delta_2= \{ (s,t)\in [0,1]^2\, ;\, s<t\}$, $u\in \R^d\setminus\{0\}$ and $\delta_0$ denotes the Dirac delta function, the $2$-fold self-intersection local time is defined by
$$ \rho_2 (u)=\lim_{\varepsilon \to 0} \int_{\Delta_2  }  p^d_{\varepsilon}\big( w(t)-w(s)-u\big)dsdt   $$
where 
$$ p^d_{\ve}(z)=\frac{1}{ (2\pi \ve)^{d/2} } \exp \left( \frac{-\|z\|^2}{2\ve} \right),\; z\in \R^d,\, \ve >0
$$
($\|\cdot\|$ denotes the Euclidean norm in $\R^d$) and the limit is understood in the sense of the Sobolev space $\D^{2,\gamma}$ where $\gamma < 2-\frac{d}{2}$. We recall that, for any real number $\gamma$, the space $\D^{2,\gamma}$ is defined as the completion of 
$$ \left\lbrace   \sum_{k=0}^n I_k (f_k) \in L^2\big( \Omega, \sigma(w),\mathbb{P}\big),\, n\in \N\right\rbrace  $$
with respect to the norm  
$$ \Big \| \sum_{k=0}^n I_k (f_k)\Big \| _{2,\gamma} ^2=\sum_{k=0}^{n} (k+1)^{\gamma } \, \E \big[ I_k(f_k)^2 \big] \,.  $$
Here $\sigma(w)$ denotes the $\sigma$-field generated by the Brownian motion and $I_k(f_k)$ denotes the $k$-multiple It\^{o} stochastic integral of the deterministic, symmetric and square integrable kernel $f_k :[0,1]^k \to \R$ (see e.g. \cite{IPV95}). When the "differentiablity index" $\gamma$ is negative the Sobolev space $\D^{2,\gamma}$ contains strictly the space $L^2\big( \Omega, \sigma(w),\mathbb{P}\big)$ of square integrable Wiener random variables and its elements are called generalised Wiener functions. For more details about the Sobolev spaces $\D^{2,\gamma}$ and the generalised Wiener functions we refer to \cite{Mall97, Su88}. Particularly, $\rho_2(u)$ is a generalised Wiener function. Moreover, it is positive. We recall that a generalised Wiener function $\eta \in \D^{2,-\gamma}$ (for some $\gamma >0$) is positive if for any Wiener test function $F\in \D^{2,\gamma}$ (we recall also that $\D^{2,-\gamma}$ is the dual space of $\D^{2,\gamma}$) we have 
$$ F(\omega) \geqslant 0 \;\mu-\text{a.e.} \, \omega \in W_0^d \; \Rightarrow \; (\eta, F)\geqslant 0. $$
Here $ W_0^d=\left\lbrace \omega :[0,1]\to \R^d,\, \omega \text{ is continuous},\, \omega(0)=0\right\rbrace  $ denotes the Wiener space which is equipped with the sigma-field $ \mathcal{B}(W_0^d) $ generated by the supremum norm and by the Wiener measure  $\mu$. The $2$-fold self-intersection local time $\rho_2(u)$ is represented by a finite positive measure $\theta_u$ on the classical Wiener space $(W_0^d,\mathcal{B}(W_0^d),\mu)$, see Theorem \ref{thm0} in Appendix.\\

In \cite{FHSW97} (see also \cite{BOS16}) the authors considered the same formal expression \eqref{example1} but at the origin, i.e. $u=0$, and for any dimension $d$, in the framework of white noise analysis. They proved that if $\rho_2(0)$ is suitably renormalised then it can be considered as a generalised white noise functional. To do this the authors first introduce the Hilbert space 
$$ \big( L^2 \big)=L^2\Big( S^*\big( \R,\R^d\big), \mu_d\Big)  $$ 
where $S^*\big( \R,\R^d\big)$ is the topological dual of the Schwartz space $S\big( \R,\R^d\big)$ and $\mu_d$ is the white noise measure, see \cite{Kuo95}.
Denote by $\Gamma(A)$ the second quantization operator (see e.g. \cite{Dor94}), acting on $\big(L^2\big)$, of 
$$ (Af)_i(t)=\Big( -\dfrac{\mathrm{d}^2}{\mathrm{d}t^2} + t^2+1 \Big ) f_i(t),\quad i=1,\cdots,d,\; (f_1,\cdots,f_d)\in S\big( \R,\R^d\big). $$
For $k\in \N$, denote by $(\mathcal{S})_k$ the domain of $\Gamma \big( A^k \big)$. Now if we denote by $\big(\mathcal{S}\big)$ the projective limit of the sequence of Hibert spaces $\{ \big(\mathcal{S}\big)_k \}$ then its topological dual $\big(\mathcal{S}\big)^*$ is the space of generalised white noise functionals and we have a Gelfand triple 
$$ \big(\mathcal{S}\big) \subset \big( L^2 \big)\subset \big(\mathcal{S}\big)^*.$$
The authors proved (see \cite[Proposition 1]{FHSW97}) that whenever $s\neq t$, the expression $ \delta_0 \big(w(t)-w(s)\big) $ is a generalised white noise functional. Next they proved (see \cite[Theorem 2]{FHSW97}) that for any integer $N$ such that $2N>d-2$ the Bochner integral 
$$L^{(2N)}=\int_{\Delta_2} \delta_0^{(2N)}(w(t)-w(s))ds dt,$$ 
where $\delta_0^{(2N)}(w(t)-w(s))$ is the projection of $ \delta_0(w(t)-w(s)) $ on the chaos of order $k\geqslant 2N$, is a generalised white noise functional. $L^{(2N)}$ is then the renormalised 2-fold self intersection local time of the Brownian motion. But the renormalisation here leads to the loss of positivity and consequently to the loss of the representation by a positive measure.\\

\noindent Our aim is to establish a large deviation principle (LDP) for the measure $ \theta_u$. For more details about large deviation theory we refer for example to \cite{DZ09,Chen09}. In \cite{Chen09}, a large deviation result for the $k$-fold self intersection local time of $1$-dimensional Brownian motion, formally defined by  
\begin{equation}
\label{8645} T_k=\int _{\Delta_k} \prod _{j=1}^{k-1} \delta_{0}(w(t_{j+1})-w(t_j)) dt_1\cdots dt_k\,.
\end{equation}
was proved (here $k\geqslant2$ and $ \Delta_k= \{ (t_1,\cdots,t_k)\in [0,1]^k\, ;\, t_1<\cdots <t_k\} $). Precisely, if we denote by $L(1,x)$ the local time of the Brownian motion, i.e. the density of the occupation measure defined on the real line by 
$$ \nu (A)=\int_0^1 \mathbf{1}_A(w(t))dt, \; A \in \mathcal{B}(\R), $$
then $T_k$ has the following representation :
$$ T_k=\dfrac{1}{k!} \int_{-\infty}^{+\infty} L^k(1,x)dx.  $$
In \cite[Theorem 4.2.1.]{Chen09} the author proved that for every $k\in \N, k\geqslant2$,
$$ \lim _{t \to \infty } t^{-\frac{2}{k-1}} \log \PP \left\lbrace k! T_k \geqslant t \right\rbrace =-\dfrac{1}{4(k-1)} \left( \dfrac{k+1}{2} \right) ^{\frac{3-k}{k-1}} B\left( \dfrac{1}{k-1}, \dfrac{1}{2} \right),   $$
where $B(.,.)$ denotes the Beta function.\\

In the case of planar Brownian motion, a LDP was also established for the renormalised $2$-fold self intersection local time formally denoted
\begin{equation} \label{renormlised_planar}
\mathcal{T}_2=\int _{\Delta_2} \delta_0 \big(w(t)-w(s) \big)  dsdt-\E \int _{\Delta_2} \delta_0 \big(w(t)-w(s) \big)  dsdt
\end{equation}
Precisely, $\mathcal{T}_2$ can be defined through "triangular approximation" :
$$ \mathcal{T}_2=\sum_{n=0}^{\infty}\Big( \sum_{k=0}^{2^n-1} \big( \beta (A_{k,n})-\E \beta (A_{k,n})\big) \Big)$$
where the series converges in $L^2$-sense and 
$$ A_{k,n}=\Big[ \dfrac{2k}{2^{n+1}} ,  \dfrac{2k+1}{2^{n+1}} \Big]\times \Big[ \dfrac{2k+1}{2^{n+1}},\dfrac{2k+2}{2^{n+1}} \Big] \quad \forall \;k,n \in \N,$$
$$ \beta (A ) =L^2 \text{-}  \lim _{\ve \to 0} \int_{A} p_{\ve}^2(w(t)-w(s))ds dt \quad \forall \;A \subset \Delta_2. $$
In \cite[Theorem 4.3.1.]{Chen09} the author proved that 
$$ \lim _{t \to \infty } \dfrac{1}{t} \log \PP \left\lbrace \mathcal{T}_2 \geqslant t \right\rbrace =-c,   $$
where $c>0$ is the best constant of the Gagliardo-Nirenberg inequality :
$$ \|f\|_4\leqslant c \sqrt{\| \nabla f \|_2} \sqrt{\| f \|_2},\quad f\in L^2(\R^2),\, \nabla f \in L^2(\R^2). $$
Here $\|\cdot\|_2$ and $\|\cdot\|_4$ denote respectively the norms in $L^2(\R^2)$ and $L^4(\R^2)$.\\

In both the two previous results, the LDP is established for "usual" random variables, in contrast with the case of a Brownian motion in $\R^d$ where $d\geqslant 4$. As we mentioned in the beginning of this introduction, the intersection local time $\rho_2(u)$ becomes a positive generalised Wiener function and the representation by a measure will be the key result to investiagte a LDP.\\ 
Let us recall some terminology related to large deviation theory.
\begin{defn} \label{rate} \cite[Chapter 1.]{DZ09}
Let $\mathcal{X}$ be a topological space and $\mathcal{B}$ its Borel $\sigma$-field. A rate function is a lower semi-continuous mapping $I : \mathcal{X} \to [0,+\infty ]$ (such that for all $\alpha \in [0,+\infty )$ the level set $\{x \, :\, I(x)\leqslant \alpha \}$ is a closed subset of $\mathcal{X}$). A good rate function is a rate function for which all level sets are compact subsets of $\mathcal{X}$. The domain of $I$ is the set of points in $\mathcal{X}$ with  finite rate $\mathcal{D}_I=\{ x\,: \, I(x)<\infty \}$.
\end{defn}
\begin{defn}\cite[Chapter 1.]{DZ09}\label{LDP}
A family  $\{\mu_{\varepsilon}\}$ of probability measures on $(\mathcal{X},\mathcal{B})$ satisfies a large deviation principle (LDP) with a rate function $I$ if, for all $A \in \mathcal{B}$,
$$ -\inf _{x\in A^{\circ}} I(x) \leqslant \liminf _{\varepsilon \to 0} \varepsilon \log \mu _\varepsilon (A) \leqslant \limsup _{\varepsilon \to 0} \varepsilon \log \mu _\varepsilon (A) \leqslant - \inf _{x\in  \overline{A}} I(x). $$
Here $A^{\circ}$ and $ \overline{A}$ denote respectively the interior and the closure of $A$.
\end{defn}
\noindent Roughly speaking this means that, for small $\varepsilon$, $\mu _\varepsilon (A)$ "behaves like" $e^{-\frac{ \inf _{x\in A} I(x)}{ \varepsilon }  }$.\\
\noindent The definition \ref{LDP} is equivalent to the following statements :
\begin{enumerate}
\item (Upper bound) For any closed set $F\subset \mathcal{X}$
$$ \limsup_{\varepsilon \to 0}  \varepsilon \log \mu _\varepsilon (F) \leqslant -\inf_{x\in F} I(x). $$
\item (Lower bound) For any open set $G\subset \mathcal{X}$
$$ \liminf_{\varepsilon \to 0}  \varepsilon \log \mu _\varepsilon (G) \geqslant -\inf_{x\in G} I(x). $$
\end{enumerate}

The paper is organised as follows. In section \ref{sec2} we start by considering some examples of positive generalised Wiener functions. Some of such examples can be encountered in physics. Then we focus on two particular examples, the first one -- denoted $\rho_k(u_1,\cdots,u_{k-1})\,$-- is a generalisation of the Wiener function $\rho(u)$ discussed in \cite{DS23}, and the second one -- denoted $\eta\,$ -- is related to winding number of Brownian motion. We prove that these two examples are indeed positive generalised Wiener functions, and consequently they can be represented by finite positive measures on the Wiener space. In order to achieve these results we prove that positive generalised Wiener functions can be obtained as products of other "elementary" positive generalised Wiener functions. These elementary bricks are obtained as Dirac Delta functions of increments of the Brownian motion, i.e. of the form $\delta_u(w(t)-w(s))$. In Proposition \ref{prop1} we show that these simple expressions represent some generalised Wiener functions. Then, in Proposition \ref{prop2} we prove that the product of two generalised Wiener functions constructed from orthogonal subspaces of the underlying Hilbert space over which the noise is constructed is again a generalised Wiener function. These two results are the main results of section \ref{sec2}. In addition to some other results related to the measures associated to $\rho_k(u_1,\cdots,u_{k-1})$ and $\eta$, the last part of section \ref{sec2} is also devoted to explicit expressions for the dual pairing between the generalised Wiener function $\eta$ and some specific kinds of Wiener test functions. In section \ref{sec3} we establish a LDP for the measure $\theta_{u_1,\cdots,u_{k-1}}$ associated to the positive generalised Wiener function $\rho_k(u_1,\cdots,u_{k-1})$. To do this we first proceed by establishing an integral expression of $\theta_{u_1,\cdots,u_{k-1}}$ in which the Wiener measure appears in a conditional form. This result is stated in Lemma \ref{lem2}. Once this result is proved we are able to apply LDP for the Wiener measure, i.e. Schilder theorem. This enables us to prove a large deviation upper bound for the measure $\theta_{u_1,\cdots,u_{k-1}}$ in Theorem \ref{thm560}. But we fail to obtain a lower bound when we replace closed sets by open sets. In this context our results are comparable to those of V.A. Kuznetsov concerning winding angle of planar Brownian motion, see \cite{Kuz15}. Section \ref{Appendix} is an Appendix. In subsection \ref{App1} we recall some results about the construction of the measure $\theta_u$, its support and an associated integral formula. In section \ref{sec2}, inspired from that case, we obtain similar results for the measure $\theta_{u_1,\cdots,u_{k-1}}$ (Theorems \ref{pgwf_measure}, \ref{thm109} and \ref{thm110}). In subsection \ref{App2}, we recall the LDP for the Wiener measure i.e. Schilder theorem. Finally, subsection \ref{FAC} contains the definition of finite absolute continuity of measures in Banach spaces and a result related to the measures associated with positive generalised Wiener functions. This notion is important to describe the relation between the measure $\theta_{u_1,\cdots,u_{k-1}}$ and the Wiener measure (Proposition \ref{fac}).

\section{Examples of positive generalised Wiener functions}
\label{sec2}

The self-intersection local time $\rho_2(u)$ formally given by \eqref{example1} is our first example of positive generalised Wiener functions. In polymer models it plays an important role. A polymer is a long molecule consisting of small blocks called monomers which are tied with chemical bonds. Examples of polymers include -but are not limited to- proteins, DNA, RNA which are natural polymers and polyethylene which is a synthetic polymer. The Edwards model represents polymers by trajectories of the Brownian motion. The Hamiltonian of this model is proportional to the self intersection local time $\rho_2(0)$ : 
$$ H=-\beta \int _{\Delta_2} \delta_0 \big(w(t)-w(s) \big)  dsdt. $$
Here the constant $\beta >0$ is called \textit{the strength of self repellence}, see \cite[Chapter 3.]{Holl09}. Moreover, the positive generalised Wiener function $\rho_2(u)$ can be encountered in constructive quantum field theory, see e.g. \cite{Sym65}.\\
In addition to \eqref{example1} one can define several similar positive generalised Wiener functions. In \cite{DI19}, the authors defined the following intersection local time 
\begin{equation} \label{89}
I^x=\int_{\Delta_2} \delta_0 \big( x_2(t_2)-x_1(t_1)\big)dt_1dt_2:=L^2\text{-} \lim_{\ve \to 0} \int_{\Delta_2} p_\ve ^2 \big( x_2(t_2)-x_1(t_1)\big)dt_1dt_2 
\end{equation} 
where $x(t)=\big( x_1(t),x_2(t)\big),\, t\geqslant 0$ is a diffusion process in $\R^4$ ($x_1$ and $x_2$ are the projections of $x$ on the plane) with generator of the form $L=\nabla \cdot (A\nabla)$. Here $A : \R^4\to \R^{4\times 4}$ is a matrix-valued function which is symmetric, smooth with bounded derivatives of all orders and such that there exists a constant $\lambda \in (0,1)$ such  that for any $y,z \in \R^4$,
$$ \lambda \|y\|^2 \leqslant \big( A(z)y,y\big) \leqslant \dfrac{\|y\|^2}{\lambda }\,. $$
The integral in \eqref{89} measures the intersections of the two $2$-dimensional curves $x_1$ and $x_2$. Therefore, it can be associated with the winding number (see e.g. \cite[Chapter 3.]{AK98}) of $x_1$ and $x_2$. The winding number of two planar curves is the number of times that one curve winds around the other. It can be seen as the linking number (see e.g. \cite[Chapter 3.]{AK98}) of the two $3$-dimensional curves $t\mapsto \big(x_j(t),t\big),\, j=1,2$. The linking number of two $3$-dimensional disjoint closed curves $\gamma_1,\gamma_2 : [0,1]\to \R^3$ can be defined by the Gauss integral (see e.g. \cite[Chapter 3.]{AK98}) 
\begin{equation}
lk\big( \gamma_1,\gamma_2\big) =\dfrac{1}{4\pi} \int_0^1 \int_0^1 \dfrac{\big(\gamma_1'(t_1),\gamma_2'(t_2),\gamma_1(t_1)-\gamma_2(t_2)\big) }{\| \gamma_2(t_2)-\gamma_1(t_1)\|^3} dt_1dt_2
\end{equation}
where $(u,v,w)$ denotes the determinant constructed from the three vectors $u,v,w\in \R^3$. The linking number is intimitely related to the helicity of a vector field, see \cite[Theorem 4.4]{AK98}. Therefore, linking number is a key notion in hydrodynamics, meteorology and electromagnetics, see e.g. \cite[Chapter 1,\$ 5 ]{AK98}. \\
Positive Wiener functions can also appear in the form 
\begin{equation}\label{Eq9}
\int _{\Delta_2} \delta_0 \big(w(t_2)-w(t_1) \big) f\big( \beta_1(t_1),\beta_2(t_2)\big ) dt_1dt_2 
\end{equation}  
where $\beta_1(t),\, \beta_2(t)$ are two independent one-dimensional Brownian motions, independent from the $d$-dimensional Brownian motion $w(t)$, and $f : \R^2\to \R$ is a non-negative measurable map. A particular case of \eqref{Eq9} is the following 
\begin{equation} \label{Eq11}
 \int _{\Delta_2} \delta_0 \big(w(t_2)-w(t_1) \big) \mathbf{1}_{(0,+\infty)}\big( \beta(t_2)-\beta(t_1)\big ) dt_1dt_2
\end{equation}
The last expression may appear, for example, when we consider a $(d+1)$-dimensional Brownian motion $B(t)=(w(t),\beta(t))$ where $w$ is a $d$-dimensional Brownian motion independent from the one-dimensional Brownian motion $\beta$. If $w$ intersects itself, i.e. $w(t_2)=w(t_1)$ for some $t_1<t_2$, then this means that there are only two possible relative positions for the two points $B(t_1),\, B(t_2)$ and these two positions are related to $\beta (t_2)-\beta(t_1)$ being either positive or negative. The integral \eqref{Eq11} can be related to the winding number of Brownian motion. And it can be defined by 
\begin{equation} \label{Eq12}
 \lim_{\ve \to 0}\int _{\Delta_2}p_\ve^d \big(w(t_2)-w(t_1) \big) \mathbf{1}_{(0,+\infty)}\big( \beta(t_2)-\beta(t_1)\big ) dt_1dt_2
\end{equation}
or, more generally, for any $u\in \R^d$,
\begin{equation} \label{Eq13}
 \lim_{\ve \to 0}\int _{\Delta_2}p_\ve^d \big(w(t_2)-w(t_1)-u \big) \mathbf{1}_{(0,+\infty)}\big( \beta(t_2)-\beta(t_1)\big ) dt_1dt_2 
\end{equation}

In this section we will focus on the following particular examples of positive generalised Wiener functions.

\begin{example}\label{Example1} Let $k\in \mathbb{N},\, k\geqslant 2$ and $u_1,\cdots, u_{k-1}\in\R^d\setminus\{0\}$ where $d\geqslant 4$. Define
\begin{equation} \label{Eq66}
 \rho_k(u_1,\cdots,u_{k-1})=\int_{\Delta_k} \prod_{j=1}^{k-1}  \delta_{u_j}\big( w(t_{j+1})-w(t_j)\big)dt_1\cdots dt_k 
\end{equation}
where $w$ is a $d$-dimensional Brownian motion and, for every $u\in \R^d$, $\delta_u$ denotes the Dirac delta function at $u$ which can be defined by 
$$ \delta_u(\cdot)=\delta_0(\cdot -u). $$ 
This example is more general than the example defined by \eqref{example1} and which was considered in \cite{DS23}. For the moment, $\rho_k(u_1,\cdots,u_{k-1})$ is a formal notation. A rigorous meaning of \eqref{Eq66} will be obtained further in this section.
\end{example}

\begin{example} \label{Example2} Define
\begin{equation}\label{Eq217}
\eta=\int _{\Delta_2} \delta_u \big(w(t_2)-w(t_1) \big) f\big( \beta(t_2)-\beta(t_1)\big ) dt_1dt_2 
\end{equation}  
where $w$ is a $d$-dimensional Brownian motion independent from the one-dimensional Brownian motion $\beta$, and $u\in \R^d\setminus \{0\}$. For the moment, the function $f$ is such that $f\big( \beta(t_2)-\beta(t_1)\big )$ is a positive square integrable random variable for any $t_1<t_2$. We will try first to give an appropriate definition of $\eta$ and prove that it is indeed a positive generalised Wiener function and then establish some related properties. In the end of this section we will consider the case when the Brownian motions $w$ and $\beta$ are not independent.
\end{example}

In order to deal with the examples \ref{Example1} and \ref{Example2} we need some preliminary results.

\begin{prop}\label{prop1}
Let $u\in \R^d\setminus\{0\}$ and $0<s<t<1$. Then the limit 
$$ \delta_u(w(t)-w(s))=\lim_{\varepsilon \to 0}  p^d_{\varepsilon}\big( w(t )-w(s)-u) $$
is a generalised Wiener function that belongs to any Sobolev space $\D^{2,\gamma}$ such that $\gamma < -d/2$. Moreover,
$$ \big \| \delta_u(w(t)-w(s)) \big\| _{2,\gamma} \leqslant C p^d_{t-s}(cu) $$
where $C>0,\, c>0$ are constants that depend only on $\gamma$.
\end{prop}

\begin{proof}
For every $n\in \N$ denote by $H_n$ the $n$-th Hermite polynomial defined by
$$ H_{n}(x)=(-1)^{n} e^{\frac{x^{2}}{2}} \left( \frac{d}{dx}  \right) ^{(n)} e^{ \frac{-x^{2}}{2}},\, x\in \R.
 $$
 From \cite{IPV95} we have the following It\^{o}-Wiener expansion  
\begin{align*}
 \delta_u(w(t)-w(s)) 
&=\sum_{k= 0}^{\infty}  \sum _{n_1+\cdots+n_d=k} \prod _{1\leqslant j\leqslant d} \left\lbrace  \dfrac{1}{  n_j! }\,H_{n_j}\Big( \dfrac{w_j(t)-w_j(s)}{\sqrt{ t-s} } \Big)\,H_{n_j}\Big( \dfrac{u_j}{\sqrt{ t-s} } \Big) \right\rbrace \,  p^d_{t-s}(u)  
\end{align*} 
from which we deduce 
\begin{align*}
\Big \|  \delta_u(w(t)-w(s)) \Big \|_{2,\gamma}^2
&=\sum_{k= 0}^{\infty}   (k+1)^{\gamma} \sum _{n_1+\cdots+n_d=k} \prod _{1\leqslant j\leqslant d} \left\lbrace  \dfrac{1}{  n_j! }\,H_{n_j}^2\Big( \dfrac{u_j}{\sqrt{ t-s} } \Big) \right\rbrace \,  \Big( p^d_{t-s}(u)  \Big)^2.
\end{align*} 
Let $\alpha   \in \big ]1/4, 1/2 \big [$.\\
From \cite{IPV95} we get the following estimations : 
\begin{align*}
& \exists \, c_1=c_1(\alpha)\, ; \; \forall \; n\in \N,\, \forall \,x\in \R\; ; \dfrac{|H_n(x)|}{ \sqrt{n!} }\leqslant c_1 (n\vee 1)^{ -\frac{8\alpha -1}{12}} e^{\alpha x^2}, \quad ( n\vee 1=\max \{n,1\}), \\
& \forall \; \beta \in ]0,1[\; \exists \, c_2=c_2( \beta)\, ; \; \forall \; k\in \N\, ; \sum _{n_1+\cdots+n_d=k} \prod _{1\leqslant j\leqslant d} (n_j\vee 1)^{ -\beta } \leqslant c_2 (k\vee 1)^{ d( 1-\beta ) -1} \,.
\end{align*}
Using these inequalities we deduce the existence of constants $c_3=c_3(\alpha),c_4=c_4(\alpha) $ such that 
\begin{align*}
\Big \| \delta_u(w(t)-w(s)) \Big \|_{2,\gamma}^2 
&\leqslant  c_3 \sum_{k= 0}^{\infty}   (k+1)^{\gamma}  (k\vee 1)^{ d\big( 1-\frac{8\alpha -1}{6} \big) -1}  \exp \Big( \dfrac{2\alpha \|u\|^2}{t-s} \Big)\Big( p^d_{t-s}(u)  \Big)^2\\
&\leqslant  c_4 \sum_{k= 0}^{\infty}   (k+1)^{\gamma}  (k\vee 1)^{ d\big( 1-\frac{8\alpha -1}{6} \big) -1}\quad  \dfrac{ \exp \Big( \frac{(2\alpha-1) \|u\|^2}{t-s} \Big)}{ (t-s)^d }\,,
\end{align*}
which is finite provided that $ \gamma+d\big( 1-\frac{8\alpha -1}{6} \big)  <0 $, which leads, after letting $\alpha\to 1/2$, to the condition $\gamma < -d/2$ and finishes the proof.
\end{proof}

\begin{prop}\label{prop2} Consider two subspaces $H_1, \, H_2$ of the Hilbert space $L_2([0,1],\R^d)$ and suppose they are orthogonal. Associate to them the following $\sigma$-fields 
$$  \mathcal{A}_j=\sigma \Big( \Big\{ \int_0^1h_i(t)dw_i(t),\, i=1,\cdots,d, \, h=(h_1,\cdots, h_d)\in H_j \Big\}\Big), \, j=1,2. $$
Let $\eta _j \in \D^{2,\gamma_j},\, j=1,2,$ be two generalised Wiener functions (so $\gamma _j <0,\; j=1,2$) such that $\eta_j$ is $\mathcal{A}_j$-measurable, $j=1,2$. Then their Wick product $ :\eta_1\;\eta_2: $ is a generalised Wiener function that belongs to the Sobolev space $\D^{2,\gamma_1+\gamma_2}$. Moreover,
$$ \big \| :\eta_1\;\eta_2: \big\| _{2,\gamma_1+\gamma_2} \leqslant   \big \| \eta_1  \big\| _{2,\gamma_1 }  \big \| \eta_2 \big\| _{2,\gamma_2}\,. $$

\end{prop}

\begin{proof}
Consider the It\^{o}-Wiener expansion of $\eta_1$ and $\eta_2$ :
$$ \eta_1= \sum_{k= 0}^{\infty} I_k,\quad \eta_2= \sum_{k= 0}^{\infty} J_k $$
where for each $k$, $I_k$ and $J_k$ are $k$-multiple stochastic integrals. Due to the specific construction of the $\sigma$-fields $\mathcal{A}_j,\,j=1,2,$ each product $I_iJ_j$ represents an $(i+j)$-multiple stochastic integral. Consequently one can consider the following series of multiple stochastic integrals :
$$ \sum_{k= 0}^{\infty} \sum_{i+j=k} I_iJ_j\,. $$
We have 
\begin{align*}
& \sum_{k= 0}^{\infty} (k+1)^{\gamma_1+\gamma_2}\E \Bigg[ \Big( \sum_{i+j=k} I_iJ_j \Big)^2 \Bigg]\\
&=\sum_{k= 0}^{\infty} (k+1)^{\gamma_1+\gamma_2} \sum_{i+j=k} \E\big(I_i^2\big) \E \big(J_j ^2\big)\quad (\text{using the  independance of }\mathcal{A}_1,\, \mathcal{A}_2)\\
&\leqslant \sum_{k= 0}^{\infty}  \sum_{i+j=k} (i+1)^{\gamma_1 } \E\big(I_i^2\big) (j+1)^{ \gamma_2}\E \big(J_j ^2\big) \quad (\text{because } \gamma_1<0,\, \gamma_2<0 )\\
&=\big \| \eta_1  \big\| _{2,\gamma_1 }^2  \big \| \eta_2 \big\| _{2,\gamma_2}^2\,. 
\end{align*}
This proves the result.
\end{proof}

\begin{remark}
Proposition \ref{prop2} describes the same result provided in mathematical physics about product of Schwartz distributions, see e.g. \cite{Vlad71}.
\end{remark}

\begin{remark}
One can check that Proposition \ref{prop2} dealing with product of generalised Wiener functions still hold if we consider a product of a generalised Wiener function and a random variable, i.e. when $\eta_1\in \mathbb{D}^{2,\gamma}$ where $\gamma <0$ and $\eta_2\in \mathbb{D}^{2,0}=L^2$, see e.g. \cite{Mall97,Kuo95}.
\end{remark} 

Now let us go back to Example \ref{Example1} and give it an appropriate definition.

\begin{thm} \label{PGWF}
Let $k\in \mathbb{N},\, k\geqslant 2$ and $u_1,\cdots, u_{k-1}$ vectors from $\R^d\setminus\{0\}$. Then, the limit
$$ \rho_k(u_1,\cdots,u_{k-1})=\lim_{\varepsilon \to 0}\int_{\Delta_k} \prod_{j=1}^{k-1}  p^d_{\varepsilon}\big( w(t_{j+1})-w(t_j)-u_j\big)dt_1\cdots dt_k $$
is a positive generalised Wiener function that belongs to any Sobolev space $\D^{2,\gamma}$ such that $\gamma < -(k-1)d/2$. 
\end{thm}

\begin{proof}
Using propositions \ref{prop1} and \ref{prop2} one can deduce that for any $(t_1,\cdots,t_k)\in \Delta_k$ and every $\gamma <-(k-1)d/2 $, the limit 
$$ \lim_{\varepsilon \to 0} \prod_{j=1}^{k-1}  p^d_{\varepsilon}\big( w(t_{j+1})-w(t_j)-u_j\big) $$
exists in the space $\D^{2,\gamma}$. Moreover, if we denote this limit by $\Phi (t_1,\cdots,t_k)$ then
$$ \big \| \Phi (t_1,\cdots,t_k)  \big\| _{2,\gamma}  \leqslant C \prod_{j=1}^{k-1}  p^d_{t_{j+1}-t_j}\big( cu_j\big) $$
for some constants $C>0,\, c>0$. For every $n\in \N$ denote by $I_n(t_1, \cdots , t_k)$ the $n$th multiple integral from the It\^{o}-Wiener expansion of $\Phi(t_1, \cdots , t_k)$. Therefore,
\begin{align*}
\big \| \rho_k(u_1,\cdots,u_{k-1})  \big\| _{2,\gamma} ^2
&=\Big \| \int_{\Delta_k}  \Phi (t_1,\cdots,t_k)dt_1\cdots dt_k \Big\| _{2,\gamma} ^2\\
&= \sum_{n= 0}^{\infty} (n+1)^{\gamma }\E  \Big( \int_{\Delta_k}  I_n (t_1,\cdots,t_k) dt_1\cdots dt_k \Big)^2    \\
&= \sum_{n= 0}^{\infty} (n+1)^{\gamma }\E  \int_{\Delta_k \times \Delta_k}  I_n (t_1,\cdots,t_k)I_n (s_1,\cdots,s_k) dt_1\cdots dt_k ds_1\cdots ds_k    \\
&=  \int_{\Delta_k \times \Delta_k} \Big< \Phi (t_1,\cdots,t_k),\Phi (s_1,\cdots,s_k) \Big>_{\D^{2,\gamma}}  dt_1\cdots dt_k ds_1\cdots ds_k    \\
&\leqslant  \int_{\Delta_k \times \Delta_k} \big \| \Phi (t_1,\cdots,t_k)  \big\| _{2,\gamma} \big \| \Phi (s_1,\cdots,s_k)  \big\| _{2,\gamma} dt_1\cdots dt_k ds_1\cdots ds_k    \\
&=  \Big( \int_{\Delta_k } \big \| \Phi (t_1,\cdots,t_k)  \big\| _{2,\gamma}   dt_1\cdots dt_k  \Big)^2    \\
&\leqslant C^2  \Big( \int_{\Delta_k } \prod_{j=1}^{k-1}    p^d_{t_{j+1}-t_j}\big( cu_j\big)   dt_1\cdots dt_k  \Big)^2    \\
& <\infty
\end{align*}
where $\big< .,. \big>_{\D^{2,\gamma}}$ denotes the inner product in the Hilbert space $\D^{2,\gamma}$.
Theorem is proved.
\end{proof}

Using the representation of positive generalised Wiener functions by measures on the Wiener space (see \cite[Theorem 4.1]{Su88} or \cite[Theorem 5]{KSW95}) one can deduce the following theorem.

\begin{thm} \label{pgwf_measure}
There exists a unique finite positive measure $\theta_{u_1,\cdots,u_{k-1}}$ on the Wiener space $(W_0^d,\mathcal{B}(W_0^d) )$ such that 
$$\forall \; F\in \mathcal{FC}_b^{\infty}\big( W_0^d\big),\quad (\rho_k(u_1,\cdots,u_{k-1}), F)=\int_{ W_0^d }  F(\omega )\,\theta_{u_1,\cdots,u_{k-1} }(d\omega).$$
\end{thm}
Following the proofs of Theorems \ref{thm1} and \ref{thm2} (see \cite{DS23}) one can prove 
\begin{thm}\label{thm109}
The measure $\theta_{u_1,\cdots,u_{k-1}}$ is supported by the set 

$$ E_ {u_1,\cdots,u_{k-1} }=\left\lbrace \omega \in W_0^d\; ; \; \exists \;0\leqslant t_1<\cdots <t_k \leqslant 1,\, \omega (t_{j+1})-\omega(t_j)=u_j,\, j=1,\cdots,k-1 \right\rbrace . $$
\end{thm}

\begin{thm} \label{thm110}
Let $F \in \mathcal{FC}_b^{\infty}\big( W_0^d\big) $. Then, the following relation holds
\begin{align*}
&\int_{W_0^d  } F(\omega) \theta_{u_1,\cdots,u_{k-1}} (d\omega) \\
&= \int_{\Delta_k  } \E \Big[F\Big | w(t_2)-w(t_1)=u_1,\cdots,w(t_k)-w(t_{k-1})=u_{k-1} \Big] \prod_{j=1}^{k-1}  p^d_{t_{j+1}- t_j }( u_j)dt_1\cdots dt_k .
\end{align*}
\end{thm}

Now let us focus on Example \ref{Example2}.

\begin{thm} \label{thm77}
Let the function $f$ be positive and such that 
$$ \int_{\R} f^2(x)e^{-\frac{x^2}{2}}dx <\infty. $$
Then, the limit
$$ 
\lim_{\varepsilon \to 0} \int_{\Delta_2  }  p^d_{\varepsilon}\big( w(t_2)-w(t_1)-u\big)f\big( \beta(t_2)-\beta(t_1)\big )    dt_1dt_2 $$
exists in every Sobolev space $\mathbb{D}^{2,\gamma}$ such that $\gamma <-\frac{d}{2} $, and is denoted by 
$$ \eta=\int _{\Delta_2} \delta_u \big(w(t_2)-w(t_1) \big) f\big( \beta(t_2)-\beta(t_1)\big ) dt_1dt_2 .$$
\end{thm}

\begin{proof}
Using Propositions \ref{prop1} and \ref{prop2} one can deduce that for any $(t_1,t_2)\in \Delta_2$ and every $\gamma <-d/2 $, the limit, which will be denoted later by $\eta(t_1,t_2)$,
$$ \lim_{\varepsilon \to 0} p^d_{\varepsilon}\big( w(t_{2})-w(t_1)-u \big) f\big( \beta(t_2)-\beta(t_1)\big )  $$
exists in the Sobolev space $\D^{2,\gamma}$ and that 
$$ \big \| \eta (t_1,t_2)  \big\| _{2,\gamma}  \leqslant C p^d_{t_{2}-t_1}( cu) \times \big \| f\big( \beta(t_2)-\beta(t_1)\big )  \big\| _{2,0}  $$
for some constants $C>0,\, c>0$. But one can check that 
\begin{align*}  
\big \| f\big( \beta(t_2)-\beta(t_1)\big )  \big\| _{2,0} 
&=\sqrt{ \int_{\R} f^2(x)p_{t_2-t_1}^1(x)dx }\\
& \leqslant \dfrac{1}{(t_2-t_1) ^{ \frac{1}{4} } }  \sqrt{ \int_{\R} f^2(x)p_{1}^1(x)dx } =\dfrac{C_1}{(t_2-t_1) ^{ \frac{1}{4} } }    
\end{align*}
for some constant $C_1$. For every $n\in \N$ denote by $I_n(t_1,t_2)$ the $n$th multiple integral from the
It\^{o}-Wiener expansion of $\eta(t_1, t_2)$. Therefore,\\
\begin{align*}
\big \| \eta \big\| _{2,\gamma} ^2
&=\Big \| \int_{\Delta_2}  \eta (t_1,t_2)dt_1dt_2\Big\| _{2,\gamma} ^2\\
&= \sum_{n= 0}^{\infty} (n+1)^{\gamma }\E  \Big( \int_{\Delta_2}  I_n(t_1,t_2) dt_1 dt_2 \Big)^2    \\
&= \sum_{n= 0}^{\infty} (n+1)^{\gamma }\E  \int_{\Delta_2 \times \Delta_2}  I_n (t_1,t_2)I_n (s_1,s_2) dt_1 dt_2 ds_1ds_2    \\
&=  \int_{\Delta_2 \times \Delta_2} \Big< \eta (t_1,t_2),\eta (s_1,s_2) \Big>_{\D^{2,\gamma}}  dt_1 dt_2 ds_1ds_2     \\
&\leqslant  \int_{\Delta_2 \times \Delta_2} \big \| \eta (t_1,t_2)  \big\| _{2,\gamma} \big \| \eta (s_1,s_2)  \big\| _{2,\gamma} dt_1 dt_2 ds_1ds_2     \\
&=  \Big( \int_{\Delta_2 } \big \| \eta (t_1,t_2)  \big\| _{2,\gamma}   dt_1dt_2  \Big)^2    \\
&\leqslant C_2  \Big( \int_{\Delta_2 }  \dfrac{p^d_{t_{2}-t_1}( cu)  }{(t_2-t_1) ^{ \frac{1}{4} } }  dt_1dt_2   \Big)^2   <\infty
\end{align*}
where $\big< .,. \big>_{\D^{2,\gamma}}$ denotes the inner product in the Hilbert space $\D^{2,\gamma}$.
Theorem is proved.
\end{proof}

\begin{thm}\label{99}
Consider  the positive generalised Wiener function defined in Theorem \ref{thm77}
$$ \eta=\int _{\Delta_2} \delta_u \big(w(t_2)-w(t_1) \big) f\big( \beta(t_2)-\beta(t_1)\big ) dt_1dt_2 .$$
Then for every $F_1\in \mathcal{FC}_b^{\infty} \big( W_0^d \big)$ and every $F_2\in \mathcal{FC}_b^{\infty} \big( W_0^1 \big)$ we have 
$$ \big( \eta, F_1(w)F_2(\beta) \big) = \int_{\Delta_2  } \E \Big[F_1(w)\Big | w(t_2)-w(t_1)=u\Big]p_{t_2-t_1}^d(u)\E \Big[F_2(\beta)f\big( \beta(t_2)-\beta(t_1)\big )\Big] dt_1dt_2. $$
\end{thm}
 
\begin{proof}
Using dominated convergence theorem, we obtain
\begin{align*}
\big( \eta, F_1(w)F_2(\beta) \big)
&= \lim_{\ve \to 0} \E \Big[F_1(w)F_2(\beta)\int_{\Delta_2  } p_{t_2-t_1+\ve}^d\big(w(t_2)-w(t_1) -u\big) f\big( \beta(t_2)-\beta(t_1)\big ) dt_1dt_2\Big]\\
&= \lim_{\ve \to 0}\int_{\Delta_2  }  \E \Big[F_1(w)F_2(\beta)p_{t_2-t_1+\ve}^d\big(w(t_2)-w(t_1) -u\big) f\big( \beta(t_2)-\beta(t_1)\big )\Big] dt_1dt_2\\
&= \lim_{\ve \to 0}\int_{\Delta_2  }  \E \Big[F_1(w)p_{t_2-t_1+\ve}^d\big(w(t_2)-w(t_1) -u\big)\Big]  \E \Big[F_2(\beta)f\big( \beta(t_2)-\beta(t_1)\big )\Big] dt_1dt_2\\
&= \int_{\Delta_2  }  \lim_{\ve \to 0} \E \Big[F_1(w)p_{t_2-t_1+\ve}^d\big(w(t_2)-w(t_1) -u\big)\Big]  \E \Big[F_2(\beta)f\big( \beta(t_2)-\beta(t_1)\big )\Big] dt_1dt_2\\
&=\int_{\Delta_2  } \E \Big[F_1(w)\Big | w(t_2)-w(t_1)=u\Big]p_{t_2-t_1}^d(u)\E \Big[F_2(\beta)f\big( \beta(t_2)-\beta(t_1)\big )\Big] dt_1dt_2.
\end{align*}
Theorem is proved.
\end{proof}
 
\begin{remark} We can also prove a similar formula for Wiener test function of the form 
 $$ F(w,\beta)=\lim_{N\to\infty} \sum_{k=1}^N F_k^N(w) G_k^N(\beta).  $$
 The dual pairing can be written as 
 $$ \big( \eta, F(w,\beta) \big) =\lim_{N\to \infty} \sum_{k=1}^N  \int_{\Delta_2  } \E \Big[F_k^N(w)\Big | w(t_2)-w(t_1)=u\Big]p_{t_2-t_1}^d(u)\E \Big[G_k^N(\beta)f\big( \beta(t_2)-\beta(t_1)\big )\Big] dt_1dt_2.  $$
\end{remark}
 
\begin{corollary}\label{98}
Let the function $f$ be positive and such that 
$$ \int_{\R} f^2(x)e^{-\frac{x^2}{2}}dx <\infty. $$
Then, the limit
$$ 
\eta=\lim_{\varepsilon \to 0} \int_{\Delta_2  }  p^d_{\varepsilon}\big( w(t_2)-w(t_1)-u\big)f\big( \beta(t_2)-\beta(t_1)\big )    dt_1dt_2 $$
is a positive generalised Wiener function. Consequently, there exists a unique finite positive measure $\theta_{u,f}$ on the Wiener space $(W_0^{d+1},\mathcal{B}(W_0^{d+1}) )$ such that 
$$\forall \; F\in \mathcal{FC}_b^{\infty}\big( W_0^{d+1}\big),\quad (\eta, F)=\int_{ W_0^{d+1} }  F(\omega )\,\theta_{u,f}(d\omega).$$
\end{corollary} 
 
\begin{thm}
When $u\to 0$ we have 
$$ \theta_{u,f} \Big( W_0^{d+1} \Big) = \mathcal{O} \Bigg( \dfrac{1}{ \|u\|^{d-\frac{3}{2}} }  \Bigg) .$$
\end{thm}

\begin{proof}
Applying Theorems \ref{99}, \ref{98} and the Cauchy-Schwartz inequality yields 
\begin{align*}
\theta_{u,f} \Big( W_0^{d+1} \Big) 
&=\int_{ W_0^{d+1} } \theta_{u,f}(d\omega)\\
&=(\eta,1)\\
&=\int_{\Delta_2  } p_{t_2-t_1}^d(u)\E \Big[f\big( \beta(t_2)-\beta(t_1)\big )\Big] dt_1dt_2\\
&=\int_{\Delta_2  } p_{t_2-t_1}^d(u) \Big(\int_{\R} f(x)p_{t_2-t_1}^1(x)dx \Big) dt_1dt_2\\
&\leqslant \int_{\Delta_2  } p_{t_2-t_1}^d(u) \sqrt{\int_{\R} f^2(x)p_{t_2-t_1}^1(x)dx} \, dt_1dt_2\\
&\leqslant C_1\int_{\Delta_2  } p_{t_2-t_1}^d(u) \dfrac{1}{(t_2-t_1) ^{ \frac{1}{4} } }  \, dt_1dt_2.
\end{align*}
Thus, the result follows from \cite[Proposition 1]{DS23}.
\end{proof}

In the last part of this section we will focus on the generalised Wiener function $\eta$ in the case where the Brownian motions $\beta$ and $w$ are not independent. Mainly, we will consider the following situation 
$$\begin{cases}
\beta(t)=r w_1(t)+\sqrt{1-r^2} \, z(t)\\
\displaystyle \eta=\int _{\Delta_2} \delta_u \big(w(t_2)-w(t_1) \big) f\big( \beta(t_2)-\beta(t_1)\big ) dt_1dt_2\\
\displaystyle\phantom{\eta}=\lim_{\varepsilon \to 0} \int_{\Delta_2  }  p^d_{\varepsilon}\big( w(t_2)-w(t_1)-u\big)f\big( \beta(t_2)-\beta(t_1)\big )    dt_1dt_2
\end{cases}  $$
where $r\in (0,1)$ is a fixed real number, $z(t)$ is a one-dimensional Brownian motion independent from $w(t)$ and $w_1$ is the first coordinate of $w(t)$. In what follows we will try to derive some expressions for the dual pairing between $\eta$ and some specific kind of test functions.

\begin{lem}\label{lem5}
Let $h$ be a bounded continuous function from $\R^2$ to $\R$ and $X$ a real random variable independent from $w(t_2)-w(t_1)$. Then 
$$ \lim_{\ve \to 0} \E \Big[ h\big( w(t_2)-w(t_1),X \big) p_\ve ^d\big( w(t_2)-w(t_1)-u\big)\Big]=p_{t_2-t_1}^d(u)\E \big[ h ( u,X ) \big]  . $$
\end{lem}

\begin{proof}
We compute
\begin{align*}
&\lim_{\ve \to 0} \E \Big[ h\big( w(t_2)-w(t_1),X \big) p_\ve ^d\big( w(t_2)-w(t_1)-u\big)\Big]\\
&=\lim_{\ve \to 0} \E \Bigg[  \E \Big[h\big( w(t_2)-w(t_1),X \big)\Big| w(t_2)-w(t_1)  \Big] p_\ve ^d\big( w(t_2)-w(t_1)-u\big)  \Bigg]\\
&=\lim_{\ve \to 0} \E \Big[  \E \big[h\big( x,X \big) \big] \Big|_{x= w(t_2)-w(t_1)} p_\ve ^d\big( w(t_2)-w(t_1)-u\big)  \Big]\\
&=p_{t_2-t_1}^d(u)\E \big[ h ( u,X ) \big].
\end{align*}
Lemma is proved.
\end{proof}

\begin{prop}\label{P8}
Consider a test function of the form $F=F_1(w(s_2)-w(s_1))F_2(z)$ where $F_1\in \mathcal{C}_b^{\infty} \big( \R^d \big)$, $F_2\in \mathcal{FC}_b^{\infty} \big( W_0^1 \big)$ and $(s_1,s_2)\in \Delta_2$ is fixed. Then, 
\begin{align*}
&( \eta, F )= \int_{\Delta_2  } p_{t_2-t_1}^d(u)\E \Big[F_1(w(s_2)-w(s_1))\Big | w(t_2)-w(t_1)=u\Big]\\
&\phantom{( \eta, F )= \int_{\Delta_2  }} \times \E \Big[F_2(z)f\big( \beta(t_2)-\beta(t_1)\big )\Big | w(t_2)-w(t_1)=u\Big] dt_1dt_2. 
\end{align*} 
\end{prop}

\begin{proof}
First we have a decomposition $ w(s_2)-w(s_1)=\alpha\times \big(w(t_2)-w(t_1)\big)+X  $ where $X$ is independent from $w(t_2)-w(t_1)$ and $\alpha= \dfrac{\E\big((w(t_2)-w(t_1))(w(s_2)-w(s_1))\big) }{d(t_2-t_1)} $.
It follows from Lemma \ref{lem5} that 
\begin{align*}
&\E \Big[ F_1\big( w(s_2)-w(s_1) \big) F_2(z)p_\ve ^d\big( w(t_2)-w(t_1)-u\big) f\big( \beta(t_2)-\beta(t_1)\big )\Big]\\
&=\E \Big[ F_1\big( w(s_2)-w(s_1) \big) F_2(z)p_\ve ^d\big( w(t_2)-w(t_1)-u\big) f\big( r\big( w_1(t_2)-w_1(t_1)\big)+\sqrt{1-r^2} \big( z(t_2)-z(t_1)\big)\big )\Big]\\
&\xrightarrow[ \varepsilon \to 0]{} p_{t_2-t_1}^d(u) \E \Big[ F_1\big( \alpha u+X \big) F_2(z)f\big( ru_1+\sqrt{1-r^2} \big( z(t_2)-z(t_1)\big)\big )\Big]\\
&\phantom{\xrightarrow[ \varepsilon \to 0]{}}= p_{t_2-t_1}^d(u) \E \Big[ F_1\big( \alpha u+X \big)\Big] \E \Big[ F_2(z)f\big( ru_1+\sqrt{1-r^2} \big( z(t_2)-z(t_1)\big)\big )\Big]
\end{align*}
where $u_1$ is the  first coordinate of $u$. Therefore,
$$( \eta, F )= \int_{\Delta_2  } p_{t_2-t_1}^d(u) \E \Big[ F_1\big( \alpha u+X \big)\Big] \E \Big[ F_2(z)f\big( ru_1+\sqrt{1-r^2} \big( z(t_2)-z(t_1)\big)\big )\Big] dt_1dt_2.$$
Using the identities
$$\begin{cases}
\E \Big[ F_1\big( \alpha u+X \big)\Big]=\E \Big[F_1(w(s_2)-w(s_1))\Big | w(t_2)-w(t_1)=u\Big],\\
\E \Big[ F_2(z)f\big( ru_1+\sqrt{1-r^2} \big( z(t_2)-z(t_1)\big)\big )\Big]=\E \Big[F_2(z)f\big( \beta(t_2)-\beta(t_1)\big )\Big | w(t_2)-w(t_1)=u\Big]
\end{cases}$$ 
leads us to the desired result.
\end{proof}

Proposition \ref{P8} can be extended as follows.

\begin{thm}
Consider a test function of the form $F=F_1(w(s_1),\cdots,w(s_n))F_2(z)$ where $F_1\in \mathcal{C}_b^{\infty} \big( \R^{d\times n} \big)$, $F_2\in \mathcal{FC}_b^{\infty} \big( W_0^1 \big)$ and $0\leqslant s_1<\cdots <s_n \leqslant 1$ are fixed. Then, 
\begin{align*}
&( \eta, F )= \int_{\Delta_2  } p_{t_2-t_1}^d(u)\E \Big[F_1(w(s_1),\cdots,w(s_n))\Big | w(t_2)-w(t_1)=u\Big]\\
&\phantom{( \eta, F )= \int_{\Delta_2  }} \times \E \Big[F_2(z)f\big( \beta(t_2)-\beta(t_1)\big )\Big | w(t_2)-w(t_1)=u\Big] dt_1dt_2. 
\end{align*} 
\end{thm}

\section{Large deviation principle for the measure $\theta_{u_1,\cdots,u_{k-1}}$}
\label{sec3}
In order to formulate a LDP for the measure $\theta_{u_1,\cdots,u_{k-1}}$ let us first find an explicit expression for this measure. In the following Lemma we state a primary expression for the measure $\theta_u$ which is the primitive form of the measure $\theta_{u_1,\cdots,u_{k-1}}$. 

\begin{lem}\label{lem1}
Let $0<t_1<\cdots<t_n<1$ and $A$ a closed subset of $\R^{nd}$. Define a closed cylindrical subset of $W_0^d$  
$$ C_{t_1,\cdots,t_n,A}=\{ \omega \in W_0^d\, ;\, (\omega (t_1),\cdots,\omega(t_n))\in A\}. $$
Then 
$$\theta_u\big(C_{t_1,\cdots,t_n,A}\big)=\int_{\Delta_2  } \lim_{\ve \to 0} \dfrac{ \mu \big\{ \omega \in W_0^d\, :\,  (\omega (t_1),\cdots,\omega(t_n))\in A, \, \omega(s_2)-\omega(s_1)\in B(u,\ve)\big\} }{ \lambda (B(u,\ve)) } ds_1ds_2 $$
where $B(u,\ve)$ is the ball, in $\R^d$, with center $u$ and radius $\ve$ and $\lambda$ is the Lebesgue measure. The limit inside the integral coincides with the density of the measure $\nu_{s_1,s_2,t_1,\cdots,t_n,A}$ defined on $\mathcal{B}(\R^d)$ by 
\begin{align*}
\nu_{s_1,s_2,t_1,\cdots,t_n,A} (M)
&=\mu \big\{ \omega \in W_0^d\, ;\, (\omega (t_1),\cdots,\omega(t_n))\in A , \, \omega(s_2)-\omega(s_1)\in M\big\}\\
&=\PP \big( (w (t_1),\cdots,w(t_n))\in A , \, w(s_2)-w(s_1)\in M\big). 
\end{align*} 
\end{lem}

\begin{proof} Let $\delta>0$ and $\phi_{\delta} \in \mathcal{C}_b^{\infty} (\R^{nd}) $ such that 
$$ \mathbf{1}_{A} \leqslant \phi_\delta \leqslant \mathbf{1}_{A_\delta} $$
where $A_{\delta}=\cup _{x\in A} B_{(x, \delta)}$ is the $\delta$-neighborhood of $A$. Then, let $F_{\delta}(\omega)=\phi_{\delta} (\omega (t_1),\cdots, \omega(t_n) )$. It follows that 
$$ \mathbf{1}_{C_{t_1,\cdots,t_n,A} } \leqslant F_\delta \leqslant \mathbf{1}_{C_{t_1,\cdots,t_n,A_{\delta}}}. $$
Consequently, the family $\big( F_\delta \big) _{\delta >0}$ converges pointwise, as $\delta \to 0$, to $\mathbf{1}_{C_{t_1,\cdots,t_n,A} }$. Thus, using the dominated convergence theorem and Theorem \ref{thm2}, we deduce that 
\begin{align*}
\theta_u(C_{t_1,\cdots,t_n,A})
&=\int_{W_0^d  } \mathbf{1}_{C_{t_1,\cdots,t_n,A}}(\omega) \theta_{u} (d\omega) \\
&=\lim _{\delta \to 0} \int_{W_0^d  } F_{\delta}(\omega) \theta_{u} (d\omega)\\
&=\lim _{\delta \to 0} \int_{\Delta_2  } \E \big(F_{\delta}(w)\big | w(t)-w(s)=u\big)p_{t-s}^d(u)dsdt \,.
\end{align*}
Following \cite{DS23}, the conditional expectations $\E \big(F_{\delta}(w)\big | w(t)-w(s)=u\big)$ and $\E \big(\mathbf{1}_{C_{t_1,\cdots,t_n,A}}(w)\big | w(t)-w(s)=u\big)$ possess continuous versions. Using again the dominated convergence theorem we deduce
$$ \E \big(F_{\delta}(w)\big | w(t)-w(s)=u\big) \xrightarrow[ \delta \to 0]{} \E \big(\mathbf{1}_{C_{t_1,\cdots,t_n,A}}(w)\big | w(t)-w(s)=u\big). $$
For $(s,t)\in \Delta_2$ define a measure $ \nu_{W_0^d,s,t} $ on $\mathcal{B}(\R^d)$ by 
$$ \nu_{W_0^d,s,t}(M)=\int_M p_{t-s}^d(x)dx. $$
Now, with a last use of the dominated convergence theorem we obtain 
\begin{align*}
\theta_u(C_{t_1,\cdots,t_n,A})
&=\int_{\Delta_2  } \E \big(\mathbf{1}_{C_{t_1,\cdots,t_n,A}}(w)\big | w(t)-w(s)=u\big)p_{t-s}^d(u)dsdt \\
&=\int_{\Delta_2  } \dfrac{ d\nu_{s,t,t_1,\cdots,t_n,A}}{d\nu_{W_0^d,s,t}} (u)p_{t-s}^d(u)dsdt \\
&=\int_{\Delta_2  } \dfrac{ d\nu_{s,t,t_1,\cdots,t_n,A}}{d\lambda} (u)dsdt \\
&=\int_{\Delta_2  } \lim_{\ve \to 0} \dfrac{ \mu \big\{ \omega \in W_0^d\, :\,  (\omega (t_1),\cdots,\omega(t_n))\in A, \, \omega(t)-\omega(s)\in B(u,\ve)\big\} }{ \lambda (B(u,\ve)) } dsdt 
\end{align*}
which proves the result.
\end{proof}

\noindent The following lemma extends the previous lemma to another family of closed subsets of the Wiener space.

\begin{lem}\label{lem200}
Let $\mathbb{T}$ be a closed subset of $[0,1]$ and $A$ a closed subset of $\R^{d}$. Consider the closed subset of $W_0^d$ defined by 
$$ C_{\mathbb{T},A}=\{ \omega \in W_0^d\, ;\, \omega (t)\in A \; \forall \; t\in \mathbb{T}\}. $$
Then 
$$\theta_u\big(C_{\mathbb{T},A}\big)=\int_{\Delta_2  } \lim_{\ve \to 0} \dfrac{ \mu \big\{ \omega \in W_0^d\, :\,  \omega (t)\in A \; \forall \; t\in \mathbb{T}, \, \omega(s_2)-\omega(s_1)\in B(u,\ve)\big\} }{ \lambda (B(u,\ve)) } ds_1ds_2. $$
The limit inside the integral coincides with the density of the measure $\nu_{s_1,s_2,\mathbb{T} ,A}$ defined on $\mathcal{B}(\R^d)$ by 
\begin{align*}
\nu_{s_1,s_2,\mathbb{T},A} (M)
&=\mu \big\{ \omega \in W_0^d\, ;\, \omega (t)\in A \; \forall \; t\in \mathbb{T} , \,\omega(s_2)-\omega(s_1)\in M\big\}\\
&=\PP \big( w (t)\in A \; \forall \; t\in \mathbb{T} , \, w(s_2)-w(s_1)\in M\big). 
\end{align*} 
\end{lem}

\begin{proof} Let $\{t_n,\,n\in \N^*\}$ be a countable dense subset of $\mathbb{T}$. For each $n\in \N^*$ define a closed cylindrical subset 
$$  C_{t_1,\cdots,t_n,A^n}=\{ \omega \in W_0^d\, ;\, \omega (t_j)\in A \; \forall \; j\in \{ 1,\cdots,n\} \}. $$
Following the proof of Lemma \ref{lem1}, for each $n\in \N^*$ there exists $F_n\in \mathcal{FC}_b^{\infty}\big( W_0^d\big) $ such that 
$$ \mathbf{1}_{C_{t_1,\cdots,t_n,A^n} } (\omega)\leqslant F_n (\omega)\leqslant \mathbf{1}_{C_{t_1,\cdots,t_n,( A_{1/n } )^n}}(\omega) \quad \forall \; \omega \in W_0^d. $$
Using the continuity of the elements $\omega$ of $W_0^d$ one can check that 
$$ \lim _{n\to\infty} \mathbf{1}_{C_{t_1,\cdots,t_n,A^n} } (\omega)= \lim _{n\to\infty} \mathbf{1}_{C_{t_1,\cdots,t_n,(A_{1/n })^n} } (\omega)=\mathbf{1}_{C_{\mathbb{T},A} } (\omega) \quad \forall \; \omega \in W_0^d. $$
Consequently, the sequence $( F_n)$ converges pointwise, as $n \to \infty$, to $\mathbf{1}_{C_{\mathbb{T},A} }$. Then we proceed in the same way as in the proof of Lemma \ref{lem1}. 
\end{proof}

Following the proof of Lemmas \ref{lem1} and \ref{lem200}, one can prove the following result.

\begin{lem}\label{lem2} Let $\mathbb{T}$ be a closed subset of $[0,1]$ and $A$ a closed subset of $\R^{d}$. Consider the closed subset of $W_0^d$ defined by 
$$ C_{\mathbb{T},A}=\{ \omega \in W_0^d\, ;\, \omega (t)\in A \; \forall \; t\in \mathbb{T}\}. $$
Then
\begin{align*}
&\theta_{u_1,\cdots,u_{k-1}}\big(C_{\mathbb{T},A}\big)=\\
&\int_{\Delta_k  } 
  \lim_{\ve \to 0} \dfrac{ \mu \big\{ \omega \in W_0^d\, :\,  \omega (t)\in A \; \forall \; t\in \mathbb{T}, \, \omega(t_{j+1})- \omega(t_j)\in B(u_j,\ve) ,j=1,\cdots,k-1\big\} }{\prod_{j=1}^{k-1} \lambda (B(u_j,\ve)) } dt_1\cdots dt_k  \,.
\end{align*} 
\end{lem}

\begin{remark}\label{rem10}
From Lemma \ref{lem2} on can deduce that for any closed subset $\mathbb{T}$ of $[0,1]$ and every closed $A\subset \R^d$ we have 
$$ \mu \big( C_{\mathbb{T},A} \big)=0 \Rightarrow \theta_{u_1,\cdots,u_{k-1}}  \big( C_{\mathbb{T},A} \big)=0. $$
But $\theta_u$ is not absolutely continuous with respect to Wiener mesure. Instead it is finitely absolutely continuous. The definition of finite absolute continuity and the main result regarding our considerations are presented in subsection \ref{FAC} of Appendix.
\end{remark}

\noindent The following proposition is an immediate consequence of Theorems \ref{PGWF}, \ref{pgwf_measure} and \ref{Ryabov}.
\begin{prop}\label{fac} Let $k\in \mathbb{N},\, k\geqslant 2$ and $u_1,\cdots, u_{k-1}\in\R^d\setminus\{0\}$. Then, the measure $\theta_{u_1,\cdots,u_{k-1}}$ is finitely absolutely continuous with respect to the Wiener mesure $\mu$.
\end{prop}

\begin{remark}\label{rem11}
We claim that the limit inside the integrals in Lemma \ref{lem2} can be written as $ \mu \big\{ \omega \in W_0^d\, :\,  \omega (t)\in A  \; \forall \; t\in \mathbb{T}, \, \omega(t_{j+1})- \omega(t_j)=u_j,\,j=1,\cdots,k-1\big\} .$
\end{remark}

\noindent For $0\leqslant t_1 < \cdots <t_k \leqslant 1$ define 
$$ E_{u,t_1,\cdots,t_k}=\left\lbrace \omega \in W_0^d\; : \; \omega(t_{j+1})- \omega(t_j)=u_j,\,j=1,\cdots,k-1\right\rbrace .  $$
Remark that 
$$\begin{cases}
E_{u_1,\cdots,u_{k-1}}=\left\lbrace \omega \in W_0^d\; : \; \exists \;0\leqslant t_1,\cdots,t_k \leqslant 1,\, \omega(t_{j+1})- \omega(t_j)=u_j,\,j=1,\cdots,k-1 \right\rbrace \\
\phantom{ E_{u_1,\cdots,u_{k-1}} }=\cup _{(t_1,\cdots,t_k)\in \Delta_k} E_{u_1,\cdots,u_{k-1},t_1,\cdots,t_k},\\
tE_{u_1,\cdots,u_{k-1},t_1,\cdots,t_k}=E_{tu_1,\cdots,tu_{k-1},t_1,\cdots,t_k} \; \forall \; t>0,.
\end{cases}$$
Remark also that for sets of the form $C_{\mathbb{T},A}$ and $t>0$ we have $ tC_{\mathbb{T},A}=C_{\mathbb{T},tA} $.\\

\begin{thm}\label{thm560}
The measure $\theta_{u_1,\cdots,u_{k-1}}$ yields the following large deviation upper bound. If $\mathbb{T}$ is a closed subset of $[0,1]$ and $A$ a closed subset of $\R^{d}$ then
$$ \limsup_{t \to \infty}  \frac{1}{t^2} \log \theta_{tu_1,\cdots,tu_{k-1}} \big(tC_{\mathbb{T},A}\big) \leqslant -\inf _{\phi \in  C_{\mathbb{T},A} \cap E_ {u_1,\cdots,u_{k-1} }} I(\phi) $$
where $I$ is the rate function associated with the Wiener measure. 
\end{thm}

\begin{proof}
It is enough to prove the result when $k=2$. Applying Lemma \ref{lem1} one can obtain
\begin{align*}
\theta_{tu}(tC_{\mathbb{T},A})
&=\theta_{tu}(C_{\mathbb{T},tA})\\
&=\int_{\Delta_2  } \mu \big\{ \omega \in W_0^d\, :\,  \omega (t)\in tA  \; \forall \; t\in \mathbb{T}, \,  \omega(s_2)-\omega(s_1)=tu\big\} ds_1ds_2 \\
&=\int_{\Delta_2  } \mu \big(  tC_{\mathbb{T},A} \cap E_{tu,s_1,s_2}\big) ds_1ds_2 \\
&=\int_{\Delta_2  } \mu \big(  t\big( C_{\mathbb{T},A} \cap E_{u,s_1,s_2}\big) \big)ds_1ds_2 \,.
\end{align*}
According to Theorem \ref{Schilder}
$$ \forall \; r>0 \; \exists \; M > 0\; ; \; \forall t> M, \; \mu \big(  t\big( C_{\mathbb{T},A} \cap E_{u,s_1,s_2}\big) \big) < e^{ -t^2 \big( \inf _{\phi \in  C_{\mathbb{T},A} \cap E_{u,s_1,s_2}} I(\phi)-r\big) }. $$
Consequently,
\begin{align*}
\theta_{tu}(tA)
&=\int_{\Delta_2  } \mu \big(  t\big( C_{\mathbb{T},A} \cap E_{u,s_1,s_2}\big) \big)ds_1ds_2 \\
&\leqslant \int_{\Delta_2  } e^{ -t^2 \big( \inf _{\phi \in  C_{\mathbb{T},A} \cap E_{u,s_1,s_2}} I(\phi)-r\big) }ds_1ds_2 \\
&\leqslant \int_{\Delta_2  } e^{ -t^2 \big( \inf _{(t_1,t_2)\in \Delta_2}\inf _{\phi \in  C_{\mathbb{T},A} \cap E_{u,s_1,s_2}} I(\phi)-r\big) }ds_1ds_2 \\
&=\dfrac{1}{2} e^{- t^2 \big(\inf _{\phi \in  A \cap E_{u }} I(\phi)-r\big)  }  \,.
\end{align*}
Taking $r$ to $0$ yields the result.
\end{proof}

 \begin{remark}
Remark that 
$$ \inf _{\phi \in  E_{u_1,\cdots,u_{k-1} }} I(\phi)  =\dfrac{1}{2}\Big( \sum_{j=1}^{k-1}\|u_j\|\Big)^2. $$
In order to prove it, let $\phi \in E_{u_1,\cdots,u_{k-1} }\cap H_1$. Then, there exists a square integrable function $h :[0,1] \to \R^d$ such  that $\phi(t)=\int_0^th(s)ds \; \forall \; t\in [0,1]$ and there exist real numbers $0\leqslant t_1<\cdots <t_k\leqslant 1$ such that $\int_{t_j}^{t_{j+1} }h(s)ds =u_j,\,j=1,\cdots,k-1$. Consequently,
\begin{align*}
I(\phi)=\frac{1}{2}\int_0^1\|h(s)\|^2ds 
&\geqslant \frac{1}{2}\Big(\int_0^1\|h(s)\|ds \Big)^2\\
&\geqslant \frac{1}{2}\Big( \sum_{j=1}^{k-1}\int_{t_j}^{t_{j+1}}\|h(s)\|ds \Big)^2\\
&\geqslant \frac{1}{2}\Big( \sum_{j=1}^{k-1} \Big\|\int_{t_j}^{t_{j+1}}h(s)ds \Big\| \Big)^2=\dfrac{1}{2}\Big( \sum_{j=1}^{k-1}\|u_j\|\Big)^2. 
\end{align*} 
Moreover, for a specific choice of $0=t_1<t_2<\cdots<t_k=1$, one can construct a piecewise linear function $\phi : [0,1]\to \R^d$ such that $\phi(t_1)=0,\;\phi(t_{j+1})-\phi(t_j)=u_j,\, j=1,\cdots,k-1,$ and 
$$ I(\phi)=\dfrac{1}{2}\Big( \sum_{j=1}^{k-1}\|u_j\|\Big)^2.  $$
\noindent In particular, for any closed $C\subset W_0^d$,
$$ \inf _{\phi \in  C\cap E_{u_1,\cdots,u_{k-1} }} I(\phi) \geqslant \dfrac{1}{2} \Big( \sum_{j=1}^{k-1}\|u_j\|\Big)^2\,. $$
\end{remark}

\begin{remark}
From Theorem \ref{thm560} one can see that the rate function $\widetilde{I}$ of the measure $\theta_u$ can be expressed in terms of the good rate function $I$ of the Wiener measure as 
$$\widetilde{I}(\phi)= \begin{cases} I(\phi), & \phi \in E_{u_1,\cdots,u_{k-1} } \\ +\infty, & \phi \notin E_{u_1,\cdots,u_{k-1} }.  \end{cases} $$
It follows that the level sets are given by
$$ \{ \widetilde{I} \leqslant \lambda \} = \{ I \leqslant \lambda \}\cap E_{u_1,\cdots,u_{k-1} }.  $$
And since $E_{u_1,\cdots,u_{k-1} }$ is a closed set and the level sets of the rate function $I$ are compact (according to Theorem \eqref{Schilder}) then the level sets of $\widetilde{I} $ are also compact. Therefore, according to Definition \eqref{rate}, $ \widetilde{I} $ is a good rate function.
\end{remark}

\section{Appendix} \label{Appendix}
\subsection{Some results about the measure $\theta_u$} \label{App1}
In \cite{DS23} the authors used a representation of positive generalised Wiener functions by measures on the Wiener space (see e.g. \cite[Theorem 4.1.]{Su88}) in order to prove that the $2$-fold self-intersection local time $\rho_2(u)$ can be represented by a finite positive measure $\theta_u$ on the classical Wiener space $(W_0^d,\mathcal{B}(W_0^d),\mu)$.

\begin{thm}\cite[Theorem 4]{DS23}\label{thm0} Let $d\geqslant 4$ and $u\in \R^d\setminus \{0\}$. Then,
$$\forall \; F\in \mathcal{FC}_b^{\infty}\big( W_0^d\big) ,\quad (\rho_2(u), F)=\int_{ W_0^d }  F(\omega )\,\theta_{u} (d\omega).$$
Here $\mathcal{FC}_b^{\infty}\big( W_0^d\big)$ denotes the set of Wiener test functions of the form 
$$F(\omega)=f(\ell_1(\omega),\cdots,\ell_n(\omega) )$$ 
where $f\in \mathcal{C}_b^{\infty}(\R^d)$ : the space of all bounded and infinitely differentiable functions from $\R^d$ to $\R$ with all their derivatives being bounded, and $\ell_1,\cdots,\ell_n \in \big( W_0^d \big) ^*$ : the topological dual of the Wiener space.
\end{thm}

Moreover, the following theorems were proved.
\begin{thm}\cite[Theorem 5]{DS23}\label{thm1}
Let $u\in \R^d,\,u\neq 0$. The support of the measure $\theta_u$ is included in the closed set : 
$$ E_u=\left\lbrace \omega \in W_0^d\; ; \; \exists \;0\leqslant s<t \leqslant 1,\, \omega(t)-\omega (s)=u \right\rbrace . $$
\end{thm}
\begin{thm}\label{thm2} \cite[Theorem 7]{DS23}
Let $\eta \in \mathcal{FC}_b^{\infty}\big( W_0^d\big) $. Then, for every $u\in \R^d\setminus \{0\}$,
\begin{equation} \label{integ.eq}
\int_{W_0^d  } \eta(\omega) \theta_{u} (d\omega) = \int_{\Delta_2  } \E \big(\eta\big | w(t)-w(s)=u\big)p_{t-s}^d(u)dsdt 
\end{equation}
Moreover, both sides in \eqref{integ.eq} are continuous with respect to the variable $u\in \R^d\setminus \{0\}$.
\end{thm}
\noindent We recall also that 
$$
\theta_{u} \big( W_0^d\big)= \int_{\Delta_2  } p_{t-s}^d(u)dsdt  =: m(u,d)
$$
and that $ \widetilde{\theta_u}=\dfrac{\theta_u }{m(u,d) } $ is a probability measure. 

\subsection{LDP for the Wiener measure} \label{App2}
The following is Schilder theorem which states a LDP for the Wiener measure.
\begin{thm}(Schilder)\cite[Theorem 5.2.3]{DZ09} \label{Schilder}
Consider the Wiener measure $\mu$ and for $\varepsilon >0$ let $\mu _\varepsilon =\mu \Big( \frac{1}{\sqrt{\varepsilon} } \;. \Big) $. Then, $\{\mu_\varepsilon \}$ satisfies, in the Wiener space $W_0^d$, a LDP with good rate function 
$$ I(\phi)=\begin{cases}
\frac{1}{2} \int _0^1 \|\phi '(t)\|^2dt, & \phi \in H_1\\
\infty ,& \text{otherwise}.\end{cases}  $$ 
Here $H_1$ is the space of all absolutely continuous functions with value $0$ at $0$ that possess a square integrable derivative and $ \|\cdot \| $ denotes the Euclidean norm in $\R^d$.
\end{thm}
\subsection{Finite absolute continuity} \label{FAC}
Finite absolute continuity was first introduced by A.A. Dorogovstev in \cite{Dor00} in order to compare Gaussian measures in Banach spaces.
\begin{defn}\cite[Definition 2]{Dor00}
Let $\nu_1, \nu_2$ be two probability measures on a Banach space $X$, which have finite weak moments of any order. We say that $\nu_1$ is finitely absolutely continuous with respect to $\nu_2$ if for any $n\in \N$ there exists a constant $c_n>0$ such that for any polynomial $P : \R^n\to \R$ with degree at most $n$ and any $\ell_1,\cdots, \ell_n \in X^*$ we have 
\begin{align*}
& \Big| \int _{ X } P ( \ell_1(\omega),\cdots , \ell_n (\omega) ) d\nu_1(\omega) \Big| \\
& \leqslant c_n \Big(  \int _{ X } P^2 ( \ell_1(\omega),\cdots, \ell_n (\omega) ) d\nu_2(\omega) \Big) ^{\frac{1}{2} }. 
\end{align*} 
\end{defn}

\begin{thm}\cite[Proposition 1]{Riabov11} \label{Ryabov}
Let $\nu $ be the measure on the Wiener space associated to some positive generalised Wiener function. Then $\nu$ is finitely absolutely continuous with respect to the Wiener measure.
\end{thm}
\section*{Disclosure statement}
The authors confirm that there are no relevant financial or non-financial competing interests to report.

\section*{Acknowledgment.} The authors are very grateful to the referee for the careful reading and the valuable comments that were useful for the improvement of this work.


\begin{thebibliography}{99}

\bibitem{AK98}
V.I. Arnold and B.A. Khesin, \emph{ Topological Methods in Hydrodynamics}, Springer-Verlag New York, (1998).


\bibitem{BOS16} 
J. Bornales, M. J. Oliveira and L. Streit : \emph{Chaos Decomposition and Gap Renormalization of Brownian Self-Intersection Local Times}, Reports on Mathematical Physics, 77 No 2 (2016) pp 141--152 



\bibitem{Chen09}
X. Chen, \emph{ Random Walk Intersections : Large Deviations and Related Topics}, Mathematical Surveys and Monographs, AMS, Vol. 157, (2009).



\bibitem{DZ09}
A. Dembo and O. Zeitouni, \emph{Large Deviations Techniques and Applications}, Stochastic Modelling and Applied Probability 38, Springer-Verlag Berlin Heidelberg (1998), corrected printing 2010.

\bibitem{Dor94} A.A. Dorogovtsev : \emph{ Stochastic analysis and random maps in Hilbert spaces}, VSP, 1994.


\bibitem{Dor00} 
A.A. Dorogovtsev, \emph{Measurable functionals and finitely absolutely continuous measures on Banach spaces}, Ukrainian Math. J. Vol. 52, (9),  (2000), pp. 1366--1379 

\bibitem{DI19} 
A.A. Dorogovtsev, Olga Izyumtseva, \emph{Asymptotics of intersection local time for diffusion processes}, Lith. Math. J. Vol. 59, (4), (2019), pp. 519--534 


\bibitem{DS23}
A.A. Dorogovtsev and N. Salhi, \emph{Refinements of asymptotics at zero of Brownian self-intersection local times}, Infin. Dimens. Anal. Quantum Probab. Relat. Top, 2350018 (24 pages), \url{https://doi.org/10.1142/S0219025723500182}, (2023).


\bibitem{FHSW97}  
M. De Faria, T. Hida, L. Streit and H. Watanabe, \emph{Intersection Local Times as Generalized White
Noise Functionals}, Acta Applicandae Mathematicae, Vol. 46, (1997), pp. 351-–362.



\bibitem{Holl09}
F. den Hollander, \emph{Random Polymers}, Lecture Notes in Mathematics 1974, Heidelberg Springer, (2009).


\bibitem{IPV95}  
P. Imkeller, V. Perez-Abreu and J. Vives, \emph{Chaos expansions of double intersection local time of Brownian motion in $\R ^d $ and renormalization}, Stochastic Process. Appl. Vol. 56, (1995), pp. 1--34.


\bibitem{KSW95}
Y.G. Kondrat'ev, L. Streit andW.  Westerkamp, \emph{A note on positive distributions in Gaussian analysis}, Ukr. Math. J. Vol. 47, (1995), 749–-759.


\bibitem{Kuo95}
H.H. Kuo, \emph{White Noise Distribution Theory}. CRC-Press, 1996.

\bibitem{Kuz15} 
V.A. Kuznetsov, \emph{On the large deviation principle for the winding angle of a Brownian trajectory around the origin}, Theory Stoch. Process. Vol. 20, (2), (2015), pp. 63--84. 


\bibitem{Mall97}
P. Malliavin, \emph{Stochastic Analysis}, Springer, (1997).

\bibitem{Riabov11} 
G.V. Riabov, \emph{Finite absolute continuity on an abstract Wiener space}, Theory Stoch. Process. Vol. 17, (1), (2011), pp. 100--108 



\bibitem{Su88}
H. Sugita, \emph{Positive generalized Wiener functions and potential theory over abstract Wiener spaces}, Osaka J. Math. 25 (1988), pp. 665--696.

\bibitem{Sym65}
K. Symanzik, \emph{Euclidean quantum field theory}, R. Jost (ed.), Local Quantum Theory, Academic
Press, New York, 1969.


\bibitem{Vlad71}
V.S. Vladimirov, \emph{Equations of Mathematical Physics},
 Marcel Dekker, (1971). 


\end{thebibliography}
\end{document}